\documentclass[11pt]{article}

\usepackage{amsmath,amsthm,amssymb,enumerate,framed,mdwlist,a4wide}
\usepackage{color,colortab}
\usepackage[ruled,noline,linesnumbered]{algorithm2e}
\usepackage{pgf,tikz}
\usepackage{mathrsfs}
\usetikzlibrary{arrows}
\usepackage{url}

\newcommand{\N}{\mathbb{N}}
\newcommand{\Z}{\mathbb{Z}}
\newcommand{\R}{\mathbb{R}}
\newcommand{\conv}{\operatorname{conv}}

\newcommand{\floor}[1]{\left\lfloor#1\right\rfloor}
\newcommand{\ceil}[1]{\left\lceil#1\right\rceil}

\renewcommand{\epsilon}{\varepsilon}

\newtheorem{nn}{}[section]
\newtheorem{lemma}[nn]{Lemma}
\newtheorem{theorem}[nn]{Theorem}
\newtheorem{corollary}[nn]{Corollary}
\newtheorem{proposition}[nn]{Proposition}
\newtheorem{definition}[nn]{Definition}

\newtheorem{obs}[nn]{Observation}
\newtheorem{rem}[nn]{Remark}

\newtheorem{REMARK}[nn]{Remark}
\newenvironment{remark}{\begin{REMARK}}{\end{REMARK}}

\numberwithin{equation}{section}

\allowdisplaybreaks

\title{Scanning integer points with lex-inequalities: A finite cutting plane algorithm for integer programming with linear objective}

\author{Michele Conforti\footnote{Dipartimento di Matematica ``Tullio Levi-Civita'', Universit\`a degli Studi di Padova, Italy.} 
\and Marianna De Santis\footnote{Dipartimento di Ingegneria Informatica Automatica e Gestionale, Sapienza Universit\`a di Roma, Italy} 
\and Marco Di Summa\footnotemark[1] \and Francesco Rinaldi\footnotemark[1]}

\date{}

\begin{document}

\maketitle

\begin{abstract}  We consider the integer points in a unimodular cone $K$ ordered by a lexicographic rule defined by a lattice basis.
To each integer point $x$ in $K$ we associate a family of inequalities (lex-inequalities) that defines the convex hull of the integer points 
in $K$ that are not lexicographically smaller than $x$. 
The family of lex-inequalities contains the Chv\'atal--Gomory cuts, but
does not contain and is not contained in the family of split cuts. This provides a finite cutting plane method to solve the
integer program $\min \{cx: x\in S\cap \Z^n\}$, where $S\subset \R^n$ is a compact set and $c\in \Z^n$.
We analyze the number of iterations of our algorithm.
\end{abstract}

\section{Introduction}
The area of integer nonlinear programming is rich in applications but quite challenging from a computational point of view.
We refer to the articles~\cite{belotti2013mixed,burer-letchford} for comprehensive surveys on these topics.
The tools that are mainly used are sophisticated techniques that exploit relaxations, 
constraint enforcement (e.g.,~cutting planes) and convexification of the feasible set.
Reformulations in an extended space and cutting planes for integer nonlinear programs 
have been investigated and proposed for some time, see e.g.~\cite{ceria1999convex,frangioni2006,stubbs:1999}.
This line of research mostly provides a nontrivial extension of the theory of disjunctive programming 
to the nonlinear case.  To the best of our knowledge, these
results are obtained under some restrictive conditions: typically, the feasible set is assumed to be convex 
or to contain 0/1 points only (these  cases cover some important areas of application).
\smallskip

In this paper we present a finite cutting plane algorithm for  problems of the
form
\begin{equation}\label{eq:problem}
\min \{cx: x\in S\cap \Z^n\},
\end{equation}
where $S$ is a compact subset of  $\R^n$ (not necessarily convex or connected) and $c\in \Z^n$.
Our algorithm uses a new family of cutting planes which 
 do not make any use of a description of the set $S$.
The cutting planes employed in our algorithm are obtained as follows.
We consider the integer points in a unimodular cone $K$, ordered by a lexicographic rule, associated with a lattice basis. 
To each integer point $x$ in $K$, we associate a family of inequalities (lex-inequalities) that, 
in a sense, is best possible, as it defines the convex hull of the integer points in $K$ that are not lexicographically 
smaller than $x$. Our family of cuts includes the Chv\'atal--Gomory cuts, but
it does not contain, nor is it contained in, the family of split cuts.

Our algorithm recursively solves optimization problems of the form $\min\{cx:x\in S\cap P\}$, 
where $P$ is a polyhedron, and we assume that an algorithm for problems of this type is available as a black box.
We remark that as long as this black box is available, no assumption on $S$ other than compactness is required
by our algorithm. To the best of our knowledge, this is in contrast to the rest of the literature where convexity
(or even polyhedrality) is a common assumption.

\smallskip

The cuts we introduce are {\em linear} inequalites. As the convex hull of the integer points in a bounded subset of $\R^n$ is a polytope,
a finite number of linear inequalities suffices for its characterization, and only $n$ such inequalities determine an optimal point.

Furthermore a finite number of linear inequalities suffices to describe some relevant relaxations of the convex hull of the integer points in a 
bounded set: most notably, Dadush, Dey and Vielma~\cite{dadush:2014} proved 
that the Chv\'atal--Gomory closure of a compact convex set is a polytope (whereas this is not the case for the split closure of the set).

However, nonlinear inequalities have also been successfully used to provide elegant convex hull characterizations.
We mention the work by Andersen and Jensen~\cite{andersen2013intersection} where a formula to describe the convex hull of a split 
disjunction applied to a second-order cone is provided. Their work is related to the paper by Modaresi et al.\ \cite{modaresi2016intersection},
where the authors derive split cuts for convex sets described by a single conic quadratic inequality and extend general intersection cuts 
to a wide variety of quadratic sets. Belotti et al.~\cite{belotti2013families,belotti:2017} introduce the so called disjunctive 
conic cuts studying families of quadratic surfaces intersected with two given hyperplanes. Burer and K{\i}l{\i}n\c{c}-Karzan~\cite{burer:2017}
extend the works cited above and show that the convex hull of the intersection 
of a second-order-cone representable set and a single homogeneous quadratic inequality can
be described by adding a single nonlinear inequality, defining an additional second-order-cone representable set.

From an algorithmic perspective, deriving a finite cutting plane procedure that uses a well defined family of inequalities does not seem to be 
straightforward. The oldest and 
most notable example is Gomory's finite cutting plane algorithm for integer linear 
programming over bounded sets based on fractional cuts~\cite{gomory:1958,gomory:1963}.
Other finite cutting plane algorithms (again for bounded sets) can be found 
in~\cite{armstrong1979page,bell1973cutting,bowman1970finiteness,he2017another} for integer linear 
programming and in~\cite{Lee2017} for mixed integer linear programming.

While in all the papers cited above the correctness of the algorithms is based on a specific procedure for solving the continuous relaxation, 
there are methods 
that only assume that an optimal solution of the continuous relaxation is given by a black box. This is the case for the lift-and-project method 
of Balas, Ceria 
and Cornu\'ejols~\cite{balas1993lift} for mixed 0/1 linear problems, the procedure described by Orlin~\cite{orlin1985finitely} for 0/1 integer 
linear programming,
and the algorithm presented by Neto~\cite{neto2012simple} for integer linear programming over bounded sets.

The family of cuts used in Neto's algorithm is related to ours. As it will be clarified later, the inequalities introduced in~\cite{neto2012simple}
are weaker than the lex-inequalities and, in particular, they are derived under the assumption that a box containing the set $S$ is known.

We notice that a common feature of the above papers is the (explicit or implicit) use of some lexicographic rule for the choice of an optimal solution of the 
continuous relaxation or the selection of the cut. This seems to be a key tool to prove finite convergence of this type of algorithms.

\medskip


The paper is organized as follows.
In Section~\ref{sec:mainres}, we introduce the lex-inequalities with their properties. 
In Section~\ref{sec:ip}, we present the cutting plane algorithm  and we show that it terminates in a finite number of iterations.
In Section~\ref{sec:exp}, an instance where the algorithm stops after an exponential number of iterations is provided. 
Furthermore, we compare the performance of our algorithm with a natural enumeration approach. 
A comparison with Chv\'atal--Gomory cuts and split cuts is presented in Section~\ref{sec:comp}.
Section~\ref{sec:conc} concludes the paper.

\section{Lexicographic orderings and lex-inequalities}\label{sec:mainres}

A {\em lattice basis} of $\Z^n$ is a set of $n$ linearly independent vectors $c^1,\dots,c^n\in\Z^n$ such that for every $v\in\Z^n$ 
we have that    $\lambda_1,\dots,\lambda_n\in\Z$ in
the unique expression $v=\sum_{i=1}^n\lambda_ic^i$.

The lex-inequalities that we introduce in this paper are defined for a given lattice basis of $\Z^n$. To simplify the presentation, 
we first work with the standard basis and then extend the results to general lattice bases.

We will use standard notions in the theory of polyhedra, for which we refer the reader to \cite{SchrijverBOOK}.

\subsection{Standard basis}

We consider the {\em lexicographic ordering $\prec$} associated with the standard basis 
$e^1,\dots,e^n$:  given $x^1,\,x^2\in \R^n$, $x^1\prec x^2$ if and only if $x^1\ne x^2$
and $x_i^1< x_i^2$, where $i$ is the smallest index for which $x_i^1\ne x_i^2$.
We use $\preceq$, $\succ$, $\succeq$ with the obvious meaning.

We consider the cone $K=\R^n_+=\{x\in \R^n:x_i\ge 0,\,i=1,\dots, n\}$.
Given $\bar{x}\in K\cap \Z^n$, we define
\[Q(\bar{x}):=\conv\{x\in K\cap \Z^n: x\succeq \bar{x}\},\]
where ``conv'' denotes the convex hull operator.



Given $\bar x\in K\setminus\{\mathbf0\}$, we define the {\em leading index} $\ell(\bar x)$ as the largest index $i$ such that $\bar{x}_i>0$.

\begin{lemma}\label{lem:extreme-pts}
Fix $\bar x\in K\cap\Z^n$. The set $Q(\bar x)$ is a full-dimensional polyhedron. Its vertices are precisely the following points $v^1,\dots,v^{\ell(\bar x)}$:
 for $k=1,\dots,\ell(\bar x)-1$, $v^k$ has entries
\begin{alignat*}4
v_i^k &= \bar{x}_i,&\quad&i=1,\dots,k-1\\
v_k^k &= \bar{x}_k+1\\
v_i^k &=0, &&i=k+1,\dots,n,
\end{alignat*}
and $v^{\ell(\bar x)}=\bar x$.
Furthermore, the recession cone of $Q(\bar x)$ is $K$.
\end{lemma}

%
%

\begin{proof}
Define $X:=\{x\in K\cap \Z^n:x\succeq\bar x\}$ and $X_i=v^i+(K\cap\Z^n)$ for every $i\in\{1,\dots,\ell(\bar x)\}$.
Note that $X=\bigcup_{i=1}^{\ell(\bar x)}X_i$ and therefore $\conv(X) =\conv\left(\bigcup_{i=1}^{\ell(\bar x)}\conv(X_i)\right)$.
As any rational polyhedral cone is an integral polyhedron, we have $\conv(K\cap\Z^n)=K$. 
As $v^i\in\Z^n$, this implies $\conv(X_i) = v^i +K$ for every $i$,
and thus these sets are integer translates of $K$.
Therefore $\conv(X)=\conv\{v^1,\dots,v^{\ell(\bar x)}\}+K$. This shows that $\conv(X)$ is 
a full-dimensional polyhedron with recession cone $K$ and its vertices 
are contained in $\{v^1,\dots,v^{\ell(\bar x)}\}$. It is easy to verify that $v^1,\dots,v^{\ell(\bar x)}$ are 
actually all vertices of $\conv(X)$.
\end{proof}

Let $\bar {x}\in K$ be given. For every $k\in\{1,\dots,n\}$ and $i\in\{1,\dots,k\}$ we define
\[d^k_i:=\begin{cases}
1 & \mbox{if $i=k$}\\
\bar{x}_k & \mbox{if $i=k-1$},\\
\bar{x}_k\prod_{j=i+1}^{k-1}(\bar{x}_j+1),&\mbox{if } i\le k-2.
\end{cases}\]
(Note that the $d^k_i$'s depend on the choice of $\bar x$, but we omit the dependence on $\bar x$ 
to keep notation simpler: this will never generate any ambiguity.)

For every $k\in\{1,\dots,n\}$, the {\em $k$-th lex-inequality} associated with $\bar {x}$ is the inequality
\begin{equation}\label{eq:lex-ineq}
\sum_{i=1}^{k}d^k_i x_i\ge \sum_{i=1}^kd^k_i\bar{x}_i.
\end{equation}
Note that when $\bar x_k=0$, \eqref{eq:lex-ineq} is the inequality $x_k\ge 0$.

\begin{theorem}
  \label{thm:convex-hull}
If $\bar x\in K\cap\Z^n$, then the lex-inequalities \eqref{eq:lex-ineq} for $k=1,\dots,n$
 and the inequalities $x_i\ge0$ for $i=1,\dots,n$ provide a description of the polyhedron $Q(\bar x)$.
\end{theorem}

\begin{proof}
As $K$ is the recession cone of $Q(\bar{x})$ (Lemma \ref{lem:extreme-pts}) 
and $Q(\bar x)\subseteq K$, it follows that every facet inducing inequality for $Q(\bar{x})$ 
(indeed every valid inequality) is of the type
\begin{equation}\label{eq:facet}
  \sum_{i=1}^{n} a_ix_i\ge a_0
\end{equation}
where $a_i\ge 0, i=0,\dots,n$.

Given $k\in\{1,\dots,n\}$, we let $Q_k(\bar{x})\subseteq \R^k$ denote the orthogonal projection of $Q(\bar{x})$ onto the first $k$ variables, 
and we define $\bar{x}_{[k]}:=(\bar{x}_1, \dots, \bar{x}_k)$.
It follows from the definition of lexicographic ordering that $Q_k(\bar{x})=Q(\bar{x}_{[k]})$.

Therefore the facet inducing inequalities  of $Q_k(\bar{x})$ are the facet inducing inequalities
of $Q(\bar{x})$ such that $a_j=0$ for $j=k+1,\dots, n$. 
(This can be seen, e.g., as a consequence of the method of Fourier--Motzkin to compute projections.)

As  the theorem trivially holds for $Q_1(\bar{x})$, to prove the result by induction on $n$ it suffices to characterize the facets with $a_n>0$. 
As the only
facet inducing inequality with  $a_n>0$ and  $a_0=0$ is $x_n\ge 0$, from now on  we consider a facet inducing inequality
\eqref{eq:facet} with $a_n>0$ and $a_0>0$.

Assume first that $\bar{x}_n=0$.  Then by Lemma \ref{lem:extreme-pts} we have 
that $Q(\bar{x})=Q_{n-1}(\bar{x})\times \{x_n\in \R:x_n\ge 0\}$ and we are done by induction.
Therefore we assume $\bar{x}_n>0$. Recall that, by Lemma \ref{lem:extreme-pts}, $Q(\bar{x})$ has $n$ vertices, $v^1,\dots, v^n=\bar{x}$.
\smallskip

\noindent  {\bf Claim 1:} {\em $\bar{x}$ satisfies \eqref{eq:facet} at equality.}

\noindent Since $v^k_n=0$ for $k=1,\dots,n-1$, if $\bar{x}$ does not satisfy \eqref{eq:facet} at equality, the inequality
\begin{equation*}
\sum_{i=1}^{n-1} a_ix_i+(a_n-\varepsilon)x_n\ge a_0
\end{equation*}
is valid for $Q(\bar{x})$ for some  $\varepsilon >0$. Since \eqref{eq:facet} is the sum of $\varepsilon x_n\ge 0$ and 
the above inequality, and these 
inequalities are not multiples of each other as $a_0>0$,  \eqref{eq:facet} does not induce a facet of $Q(\bar{x})$. This proves Claim 1.
\smallskip

\noindent{\bf Claim 2:} {\em $a_k>0$ for $k=1,\dots,n$.}

 \noindent   By Claim 1 we have that
    \begin{equation*}
  \sum_{i=1}^{n} a_i\bar{x}_i= a_0.
\end{equation*}
Pick $k\in\{1,\dots,n-1\}$. Since $v^k$ satisfies \eqref{eq:facet}, we have that
  \begin{equation*}
  \sum_{i=1}^{k} a_i\bar{x}_i+a_k\ge a_0.
\end{equation*}
Subtracting the above equation from this inequality, we obtain
\begin{equation*}
  a_k\ge \sum_{i=k+1}^{n} a_i\bar{x}_i>0,
\end{equation*}
where the strict inequality follows because $a_i\ge0$ for $i=1,\dots,n$ and $a_n\bar{x}_n>0$.
This proves Claim 2.
\smallskip

Claim 2 shows that  if $x''\ne x'$, $x''\ge x'$ (componentwise) and $x'$ satisfies \eqref{eq:facet}, 
then $x''$ cannot satisfy \eqref{eq:facet} at equality.
In particular, if $x'$ satisfies \eqref{eq:facet} at equality and $r$ is a nonzero ray of $Q(\bar x)$ 
then $x'+r$ cannot satisfy \eqref{eq:facet} at equality.
Therefore, as $Q(\bar{x})$ is a full dimensional polyhedron and \eqref{eq:facet} induces a facet, 
this inequality must be satisfied
at equality by $v^1,\dots, v^n$.
By imposing these $n$ equations, it can be derived that, up to positive scaling, $a_i = d_i^n$, for all $i=1,\ldots,n$
and  $a_0 = \sum_{i=1}^nd^n_i\bar{x}_i$.
This implies that \eqref{eq:facet} is
\[\sum_{i=1}^{n}d^n_ix_i\ge \sum_{i=1}^nd^n_i\bar{x}_i\]
and the theorem is proven.
\end{proof}

\begin{remark}
In the description given by Theorem \ref{thm:convex-hull}, for every $k$ such that $\bar x_k=0$ the $k$-th lex-inequality is redundant, 
as it is the inequality $x_k\ge0$.
Furthermore, if $\bar x_1>0$ then also the inequality $x_1\ge0$ is redundant, 
as it is dominated by the first lex-inequality (which is $x_1\ge \bar x_1$). It can be verified 
that the remaining inequalities provide an irredundant description of $Q(\bar x)$.
\end{remark}

\subsection{General lattice bases}

Let $\{c^1,\dots,c^n\}$ be a lattice basis of $\Z^n$. Then the $n\times n$ matrix $C$ whose rows are $c^1,\dots,c^n$ is unimodular, 
i.e., it is an integer matrix
with determinant 1 or $-1$. The unimodular transformation $x\mapsto Cx$ and its inverse map integer points to integer points. 
By applying the transformation
$x\mapsto Cx$, the results of the previous subsection can be immediately extended to the lattice basis $\{c^1,\dots,c^n\}$.

In particular, the lexicographic ordering defined by the lattice basis is as follows: given $x^1,\,x^2\in \R^n$, 
we have $x^1\prec x^2$ if and only if $x_1\ne x_2$ 
and $c^ix^1< c^ix^2$, where $i$ is the smallest index for which $c^ix^1\ne c^ix^2$.

The unimodular cone $K$ is defined as $K:=\{x\in\R^n:c^ix\ge0,\,i=1,\dots,n\}$ and, 
for $\bar x\in K\cap\Z^n$, $Q(\bar x):=\conv\{x\in K\cap\Z^n:x\succeq\bar x\}$.

The leading index $\ell(\bar x)$, for $\bar x\in K\setminus\{\mathbf0\}$, is the largest index $i$ such that $c^i\bar{x}>0$.
Lemma \ref{lem:extreme-pts} now reads as follows:

\begin{lemma}\label{lem:extreme-pts-gen}
Fix $\bar x\in K\cap\Z^n$. The convex set $Q(\bar x)$ is a full-dimensional polyhedron. 
Its vertices are precisely the following points $v^1,\dots,v^{\ell(\bar x)}$:
 for $k=1,\dots,\ell(\bar x)-1$, $v^k$ is the unique point satisfying
\begin{alignat*}4
c^iv^k &= c^i\bar{x},&\quad&i=1,\dots,k-1\\
c^kv^k &= c^k\bar{x}+1\\
c^iv^k &=0, &&i=k+1,\dots,n,
\end{alignat*}
and $v^{\ell(\bar x)}=\bar x$.
Furthermore, the recession cone of $Q(\bar x)$ is $K$.
\end{lemma}

For $\bar {x}\in K$, $k\in\{1,\dots,n\}$ and $i\in\{1,\dots,k\}$, the definition of the $d^k_i$'s is as follows:
\begin{equation}\label{eq:d}
d^k_i:=\begin{cases}
1 & \mbox{if $i=k$}\\
c^k\bar{x} & \mbox{if $i=k-1$},\\
c^k\bar{x}\prod_{j=i+1}^{k-1}(c^j\bar{x}+1),&\mbox{if $i\le k-2$}.
\end{cases}
\end{equation}
The {\em $k$-th lex-inequality} associated with $\bar {x}$ is the following:
\begin{equation}\label{eq:lex-cut-gen}
\sum_{i=1}^{k}d^k_i c^ix\ge \sum_{i=1}^kd^k_ic^i\bar{x}.
\end{equation}

Theorem \ref{thm:convex-hull} now reads as follows:

\begin{proposition}\label{prop:valIn}
If $\bar x\in K\cap\Z^n$, then the lex-inequalities \eqref{eq:lex-cut-gen}
for $k=1,\dots,n$ and the inequalities $c^ix\ge0$ for $i=1,\dots,n$ provide a description of the polyhedron $Q(\bar x)$.
\end{proposition}

Neto \cite{neto2012simple} describes a family of inequalities that, although presented in a different setting, 
can be seen to be valid for $Q(\bar x)$ when the lattice basis
$\{c^1,\dots,c^n\}$ is the standard basis. However, those inequalities in general do not 
induce facets of $Q(\bar x)$ and are therefore weaker than the lex-inequalities.
In particular, the inequalities in \cite{neto2012simple} are derived under the assumption that a box containing the continuous 
set $S$ is known, and their coefficients depend on the size of the box. 
In contrast, our inequalities only depend on the current fractional solution $\bar x$. As a consequence, we obtain inequalities
with smaller dynamism (i.e., with smaller ratio between the largest and the smallest absolute value of the coefficients), 
which is a desirable property in practice.

In order to compare Neto's inequalities with ours, let $n=2$, $\bar x=(1,1)$, and consider the box $[0,3]\times[0,3]$. Neto's inequalities are in this case $x_1\ge 1$ and $3x_1+x_2\ge 4$,\footnote{We note that Neto presents his inequalities in a different form, as he considers the integer points that are lexicographically {\em smaller} than $\bar x$.} while the lex-inequalities are $x_1\ge 1$ and $x_1+x_2\ge 2$. Since the inequality $3x_1+x_2\ge 4$ is a proper conic combination of the two lex-inequalities, it cannot be facet inducing for $Q(\bar x)$.

It should also be noted that in \cite{neto2012simple} the inequalities are described only for the case in which the continuous 
set $S$ is  a bounded polyhedron, although it is 
not difficult to extend them to the case of a compact set.

Furthermore, we remark that a linear-inequality description of the set $Q(\bar x)$ can be inferred from a result proved by 
Gupte \cite[Theorem 2]{gupte2016} in the context
of super-increasing knapsack problems: one needs to apply a change of variables and observe that by removing the lower 
bounds appearing in \cite{gupte2016} the
remaining facet inducing inequalities are unaffected.

\section{The cutting plane algorithm}\label{sec:ip}

 Let $\mathcal S$ be a family of compact (not necessarily connected or convex) subsets of $\R^n$ with the following property:

{\centerline{\em if $S\in\mathcal S$ and $H$ is a closed halfspace in $\R^n$, then $S\cap H\in\mathcal S$.}}

{\em Linear optimization} over $\mathcal S$ is the following problem:
 given  $S\in\mathcal S$ and  $c\in\Z^n$, determine an optimal solution to the problem $\min\{cx:x\in S\}$ or certify that $S=\emptyset$.
(Since $S$ is compact, either $S=\emptyset$ or the minimum is well defined.)

{\em Integer linear optimization} over $\mathcal S$ is defined similarly, but the feasible
region is $S\cap \Z^n$, the set of integer points in $S$.
\smallskip

We prove that an oracle for solving linear optimization over $\mathcal S$ suffices to design a finite cutting plane 
algorithm that solves integer linear optimization over $\mathcal S$.
\medskip

We now make this statement more precise.
Given a compact subset $S$ of $\R^n$ and $c\in \Z^n$, let $ \bar x\in S$ be an optimal solution of the program $\min\{cx:x\in S\}$. A {\em cutting plane} 
is a linear inequality that is  valid for $S\cap \Z^n$ and is violated by $\bar x$. Note that a cutting plane exists if and only if $\bar x\notin\conv(S\cap\Z^n)$. 
In particular, this is certainly the case if $\bar x$ is a non-integral extreme point of $\conv(S)$. 

A (pure) {\em cutting plane algorithm} for integer linear optimization over $\mathcal S$ is an iterative procedure of the following type:
\begin{itemize}
\item[-]Let  $S\in\mathcal S$ and $c\in\Z^n$ be given.
\item[-]If $S=\emptyset$, then $S\cap \Z^n=\emptyset$. Otherwise, find an optimal solution $\bar x$ of $\min\{cx:x\in S\}$.
\item[-]If $\bar x\in S\cap \Z^n$, stop: $\bar x$ is an optimal solution to $\min\{cx:x\in S\cap\Z^n\}$. Otherwise, detect a  cutting plane and let $H$ denote the corresponding half-space.
Replace $S$ with $S\cap H$ and iterate.
\end{itemize}

Assume without loss of generality that the objective function vector $c$ is nonzero and has relatively prime entries. Then there exists a lattice basis $\{c^1,\dots,c^n\}$ of $\Z^n$ such 
that $c^1=c$. The optimal solution $\bar x$ of $\min\{cx:x\in S\}$ found by our algorithm will be a {\em lexicographically  minimum} or {\em lex-min} solution in $S$ with 
respect to the lattice basis: i.e., $\bar{x}\prec x$ for every $x\in S\setminus \{\bar{x}\}$. The lex-min vector $\bar{x}$ in $S$ satisfies the following conditions:
\begin{itemize}
\item $c^1\bar x=\min\{c^1x:x\in S\}$;
\item $c^2\bar x=\min\{c^2x:x\in S,\, c^1x=c^1\bar {x}\}$;
\item $c^3\bar x=\min\{c^3x:x\in S,\, c^1x=c^1\bar {x},\, c^2x=c^2\bar {x}\}$;
\item \dots
\item $c^n\bar x=\min\{c^nx:x\in S,\, c^1x=c^1\bar {x},\dots,\, c^{n-1}x=c^{n-1}\bar {x}\}$.
\end{itemize}
Since  $S$ is nonempty  and compact, the above minima are well-defined and can be computed by applying the oracle $n$ times. 
Furthermore these conditions uniquely define $\bar x$. 
One verifies that $\bar x$ is an extreme point of $\conv(S)$.

\begin{algorithm}[!ht]
\label{alg:weak}
\caption{Resolution of integer linear optimization over $\mathcal S$}
\SetKwInput{Input}{Input}
\SetKwInput{Output}{Output}
  \Input{$S\in\mathcal S$ with $S\ne\emptyset$, $c\in\Z^n\setminus\{\mathbf0\}$ with relatively prime entries, and a lattice basis $\{c^1,\dots,c^n\}$ of $\Z^n$ with $c^1=c$.}
	\Output{an optimal integer solution $\bar x$ for the problem $\min\{cx:x\in S\}$  or a certificate that $S\cap \Z^n=\emptyset$.}
	Compute  $\ell^*_i:=\min\{c^ix:x\in S\}$ and $\ell_i:=\ceil {\ell^*_i}$ for $1\le i\le n$, and apply a translation so that $\ell_i=0$ for $1\le i \le n$. 
	Let $K:=\{x\in \R^n:c^ix\ge 0,\,i=1,\dots, n\}$ and replace $S$ with  $S\cap K$.\label{step:preprocessing}\\
	If $S=\emptyset$, stop: the given problem is infeasible.\label{step:empty}\\
	Else, compute the lex-min solution $\bar x$ in $S$ with respect to $\{c^1,\dots,c^n\}$.\label{step:lex-min}\\
	If $\bar x\in\Z^n$, return $\bar x$.\\
	Else, let $k$ be the smallest index such that $c^k\bar x\notin\Z$\label{step:k} and compute
	\[d^k_i:=\begin{cases}
1 & \mbox{if $i=k$}\\
\ceil{c^k\bar{x}} & \mbox{if $i=k-1$},\\
\ceil{c^k\bar{x}}\prod_{j=i+1}^{k-1}(c^j\bar{x}+1),&\mbox{if $i\le k-2$}.
\end{cases}
\]
	Replace $S$ with $S\cap H$, where $H$ is the halfspace defined by the inequality \eqref{eq:cut}
  \begin{equation*}
    \sum_{i=1}^{k} d_i^kc^ix\ge  \sum_{i=1}^{k-1} d_i^kc^i\bar x + d_k^k\ceil{c^k\bar x}
  \end{equation*}
  and go to step \ref{step:empty}. \label{step:cut}
\end{algorithm}

Algorithm \ref{alg:weak} describes the procedure in detail.
Note that since $S$ is compact,  numbers $\ell^*_1,\dots,\ell^*_n$ (as defined in Algorithm \ref{alg:weak}) 
exist  and can be determined by querying the linear optimization oracle $n$ times.
Moreover, as $\{c^1,\dots,c^n\}$ is a lattice basis of $\Z^n$, an index $k$ as in step \ref{step:k} always exists when $\bar x\notin\Z^n$.

Given $x\in K$, let $x^{\uparrow}$ be the lex-min vector in $K\cap \Z^n$ such that $x\preceq x^{\uparrow}$. Obviously $x=x^{\uparrow}$  
if and only if $x\in \Z^n$.
If $x\not\in \Z^n$, let $k$ be the smallest index such that $c^k x\not\in\Z$.  It is easy to see that  $x^{\uparrow}$
is the unique point satisfying the following conditions:
 \begin{equation}\label{eq:uparrow}
 c^ix^{\uparrow}=c^i x,\,i<k;\quad c^kx^{\uparrow}=\ceil{c^k x};\quad
 c^ix^{\uparrow}=0,\,i>k.
 \end{equation}
 
 \begin{definition}\label{def:lexcut}
  Let $\bar x\not\in S\cap \Z^n$ and let $k$ be the smallest index such that $c^k \bar x\not\in\Z$.
  The \emph{$k$-th lex-cut} is the $k$-th 
  lex-inequality  associated with $\bar x^{\uparrow}$:
 \begin{equation}\label{eq:cut}
    \sum_{i=1}^{k} d_i^kc^ix\ge  \sum_{i=1}^{k-1} d_i^kc^i\bar x + d_k^k\ceil{c^k\bar x}
  \end{equation}
	(This is the cut introduced at step \ref{step:cut} of Algorithm \ref{alg:weak}.)  
 \end{definition}

\begin{proposition}\label{prop:cut}
Inequality \eqref{eq:cut} defines a cutting plane. Algorithm \ref{alg:weak} terminates after a finite number of iterations.
\end{proposition}

\begin{proof}
Since, after the preprocessing of 
step \ref{step:preprocessing}, $S\subseteq K$ and $\bar{x}$ is the lex-min point in $S$, $\bar{x}\preceq \bar{x}^{\uparrow}\prec x'$ for 
every $x'\in S\cap \Z^n\setminus\{\bar{x}^{\uparrow}\}$. 
Thus $S\cap \Z^n\subseteq Q(\bar{x}^{\uparrow})$ and by Proposition \ref{prop:valIn} inequality \eqref{eq:cut} is valid for $S\cap \Z^n$.
As $c^k\bar{x}\not\in \Z$ and $d^k_k>0$, 
the inequality is violated by $\bar{x}$. This shows that \eqref{eq:cut} defines a cutting plane.

As different iterations of the algorithm use cuts \eqref{eq:cut} associated with lexicographically increasing vectors 
 in  $S\cap \Z^n$, and $S$ is bounded, the number of
iterations of the algorithm is finite.
 \end{proof}

We mention that Akshay Gupte (personal communication) has elaborated an algorithm to solve $\min\{cx : x\in S\}$, assuming that a 
box $B$ containing $S$ is given. His algorithm iteratively constructs the convex hull of  a set of the 
form $\{x\in B\cap\Z^n : x\succeq \hat x\}$ for some $\hat x\in \Z^n$, which can be seen as a truncated version of $Q(\hat x)$. 
However, while in Algorithm \ref{alg:weak} at each iteration we use $\hat x=\bar x^\uparrow$, where $\bar x$ is 
the optimal solution of the continuous relaxation, in Gupte's algorithm $\hat x$ is obtained by ``rounding'' the point 
optimizing an objective function with superincreasing coefficients that is different from the original objective function.
As a consequence, in Gupte's algorithm one can have $\hat x \prec \bar x^\uparrow$, which makes $Q(\hat x)$ (or its truncated version) weaker.

\section{Lexicographic enumeration and the number of iterations}\label{sec:exp}

Recall the notation $x^\uparrow$ introduced in \eqref{eq:uparrow}. We extend that definition to sets as follows:
given $S\subseteq \R^n$, let $S^{\uparrow}:=\{x^{\uparrow}:x\in S\}$.
Since $S$ is bounded, $S^\uparrow$ is a finite set, as, given $y\in S^\uparrow$ and $i\in\{1,\dots,n\}$, $c^iy$ is an integer value satisfying 
$\min\{c^ix:x\in S\}\le c^iy\le\ceil{\max\{c^ix:x\in S\}}$.

\begin{obs}\label{obs:performance}
Given a nonempty set  $S\in \mathcal S$, let $(\bar{x})$ be the sequence of points in $S$ computed at step \ref{step:lex-min} of 
Algorithm \ref{alg:weak}. 
Then the sequence $(\bar{x}^{\uparrow})$ is the lex-ordering  of some distinct points in  $S^{\uparrow}$.
\end{obs}

\begin{proof}
If $\bar x$ is a point computed at step \ref{step:lex-min} of Algorithm \ref{alg:weak}, then clearly $\bar x^\uparrow\in S^\uparrow$, as $\bar x \in S$.
Thus we only have to show that if $\bar x$ and $\tilde x$ are points computed at step \ref{step:lex-min} in two consecutive iterations (say iterations $q$ and $q+1$), 
then $\bar x^\uparrow\prec\tilde x^\uparrow$. Assume not. Then $\bar x^\uparrow=\tilde x^\uparrow$ and therefore the cuts introduced at these two iterations would be 
exactly the same. But then the cut generated at iteration $q$ would already cut off $\tilde x$, contradicting the fact that at iteration $q+1$ the point computed at 
step \ref{step:lex-min} is $\tilde x$.
\end{proof}
\begin{corollary}\label{cor:sup}
$|S^{\uparrow}|$ is an upper bound on the number of cuts 
produced by  Algorithm \ref{alg:weak}. 
\end{corollary}

We next construct a convex body containing no integer points for which the bound $|S^{\uparrow}|$ on the number of cuts is exponential and tight.

\begin{proposition}\label{prop:exp}
For every $n\in\N$, there is a convex subset $S$ of $[0,1]^n$ (described by a single convex constraint plus variable bounds) on which Algorithm \ref{alg:weak} 
computes $|S^{\uparrow}|=2^n-1$ cuts.
\end{proposition}

\begin{proof}
We choose the standard basis $\{e^1,\dots,e^n\}$ as lattice basis of $\Z^n$.
Let $\mathbf1$ be the point in $\R^n$ with all entries equal to 1, and let $\|\cdot\|$ denote the Euclidean norm.
Define
\[S:=\left\{x\in[0,1]^n:\left\|x-\frac{\mathbf1}2\right\|^2\le \frac n4-\frac3{16}\right\}.\]
Note that $S\cap\Z^n=\emptyset$ and $\ell_i=0$ for $i=1,\dots,n$.
Furthermore, for every $x\in\{0,1\}^n\setminus\{\mathbf1\}$, $S$ contains the point $z(x)$ obtained from $x$ by setting to $\frac14$ the entry with largest index 
that is 0. As $S\subseteq[0,1]^n$, this shows that $S^{\uparrow}=\{0,1\}^n\setminus \{\mathbf0\}$, and thus $|S^\uparrow|=2^n-1$.

We now show that every point in $S^{\uparrow}$ is of the form $\bar x^\uparrow$ for some point $\bar x$ found in step \ref{step:lex-min}.
Let $\bar x$ be the point computed at some iteration of step \ref{step:lex-min} and assume $\bar x^\uparrow\ne\mathbf1$. 
By Theorem \ref{thm:convex-hull}, 
the lex-cut associated with $\bar x^\uparrow$
is satisfied by all $x\in \{0,1\}^n$ such that $x\succeq \bar x^\uparrow$. As the lex-cut associated with $\bar x^\uparrow$ is an inequality 
with nonnegative coefficients, it is also satisfied by the point $z(\bar x^\uparrow)$. This implies that, if we denote by 
$\tilde x$ the point computed in step 
\ref{step:lex-min} at the next iteration, $\tilde x^\uparrow$ is the lex-min point in $\{0,1\}^n$ that is lexicographically larger 
than $\bar x^\uparrow$. Thus 
every point in $S^{\uparrow}$ is of the form $\bar x^\uparrow$ for some point $\bar x$ found in step \ref{step:lex-min}. 
Together with Observation \ref{obs:performance}, 
this shows that precisely $|S^\uparrow|$ cuts are needed to discover that $S$ contains no integer points, which happens at step 
\ref{step:empty} immediately after the
iteration in which $\bar x^\uparrow=\mathbf1$.
\end{proof}

Corollary \ref{cor:sup} gives a guarantee on the maximum number of iterations of Algorithm \ref{alg:weak}.
One may ask whether there exists an 
enumerative algorithm that achieves the same performance.
We propose Algorithm \ref{alg:enum}, which we think is the best candidate. 

\begin{algorithm}[ht]
\label{alg:enum}
\caption{Resolution of integer linear optimization over $\mathcal S$ via lex-enumeration}
\SetKwInput{Input}{Input}
\SetKwInput{Output}{Output}
  \Input{$S\in\mathcal S$, $c\in\Z^n\setminus\{\mathbf0\}$ with relatively prime entries and a lattice basis $\{c^1,\dots,c^n\}$ of $\Z^n$, with $c^1=c$.}
	\Output{an optimal integer solution $\bar x$ for the problem $\min\{cx:x\in S\}$  or a certificate that $S\cap \Z^n=\emptyset$.}
	Translate $S$ so that $S\subseteq \{x\in \R^n:c^ix\ge 0,\,i=1,\dots, n\}$.  Set $\alpha_1:=\dots :=\alpha_n:=0$ and $i^*:=1$.\\
Let $S^*:=S\cap\{x\in \R^n: c^ix=\alpha_i,\,i<i^*;\:c^ix\ge\alpha_i,\,i\ge i^*\}$\label{recursion}.\\
	If $S^*=\emptyset$:\\
	\qquad If $i^*=1$, stop: $S\cap\Z^n=\emptyset$.\label{step:i^*=1}\\
	\qquad Else update {$i^*:= i^*-1$, $\alpha_{i^*}:=\alpha_{i^*}+1$,  $\alpha_{i}:=0$ for  $i>i^*$, and go to step \ref{recursion}\label{alpha1}}.\\
  Else\\
	\qquad Let $\bar{x}$ be the lex-min point in $S^*$\label{bar x}.\\
	\qquad If $\bar{x}^{\uparrow}\in S^*$, stop:  $\bar{x}^{\uparrow}$ is the lex-min point in $S\cap\Z^n$.\\
  \qquad Else update $i^*:=n$, $\alpha_{i}:=c^i\bar{x}^{\uparrow}$ for $i=1,\dots, n$, and go to step \ref{recursion}\label{alpha2}.
\end{algorithm}

The correctness of this algorithm is based on the following lemma (whose proof also explains how the algorithm works).

\begin{lemma}\label{lem:consecutive}
In Algorithm \ref{alg:enum}, if $\bar x$ and $\tilde x$ denote the points computed at two consecutive executions of line \ref{bar x}, then $\tilde x$ is the
lex-min point in $S$ that is lexicographically larger than $\bar x^\uparrow$.
\end{lemma}

\begin{proof}
Let $\bar x$ denote the point found at some execution of step \ref{bar x}. If $\bar x^\uparrow\notin S^*$, then, at steps \ref{alpha2} and \ref{recursion}, $S^*$ 
is defined as the set of points $x\in S$ satisfying $c^ix=c^i\bar{x}^{\uparrow}$ for all $i\le n-1$ and $c^nx\ge c^n\bar{x}^{\uparrow}$.
If $S^*\ne\emptyset$ then the next execution of step \ref{bar x} yields the lex-min point in $S$ that is lexicographically larger than $\bar x^\uparrow$.
Otherwise, step \ref{alpha1} is executed, which updates $S^*$ to the set of points $x\in S$ satisfying $c^ix=c^i\bar{x}^{\uparrow}$ for all
$i\le n-2$, $c^{n-1}x\ge c^{n-1}\bar{x}^{\uparrow}+1$ and $c^nx\ge0$. Again, if $S^*\ne\emptyset$ then the next execution of step \ref{bar x} 
yields the lex-min point in $S$ that is lexicographically larger than $\bar x^\uparrow$. Otherwise, this point is found after some further executions of step 
\ref{alpha1} (unless the condition in step \ref{step:i^*=1} is satisfied).
\end{proof}

To analyze the performance of Algorithm \ref{alg:enum}, we need the following definitions.
Let $C$ be the $n\times n$ matrix whose rows are $c^1,\dots,c^n$ and let $S\in\mathcal S$ be given.
For every $\bar x \in S^\uparrow$, let $V(\bar x)$ be the set of the following $n$ vectors $\alpha^1,\dots,\alpha^n$: for $k=1,\dots,n-1$, $\alpha^k$ is defined as
\begin{alignat*}4
\alpha^k_i &= c^i\bar{x},&\quad&i=1,\dots,k-1\\
\alpha^k_k &= c^k\bar{x}+1\\
\alpha^k_i &=0, &&i=k+1,\dots,n,
\end{alignat*}
and $\alpha^n=C\bar x$.
Notice that by Lemma \ref{lem:extreme-pts-gen}, $V(\bar x)$ contains all vectors of the form $Cx$ where $x$ is a vertex of $Q(\bar{x})$.

Let $V(S)=\bigcup_{\bar x\in S^\uparrow}V(\bar x)$. Notice that, given $\bar x,\bar y\in S^\uparrow$, the set $V(\bar x)\cap V(\bar y)$ may be nonempty.

\begin{proposition}\label{prop:bound-alg2}
Given a set $S\in \mathcal S$, let $(\alpha)$ be the sequence of vectors used to define the sequence of sets $(S^*)$ in step \ref{recursion} of Algorithm \ref{alg:enum}.
\begin{itemize}
\item If $S\cap\Z^n=\emptyset$, then $(\alpha)$ is the lex-ordering  of all points in $V(S)\cup\{\mathbf0\}$ with respect to the standard basis.
\item If $S\cap\Z^n\ne \emptyset$, the sequence is truncated to the lex-min vector $\alpha$ (with respect to the standard basis) such that $C^{-1}\alpha\in S\cap\Z^n=S\cap S^{\uparrow}$.
\end{itemize}
\end{proposition}

\begin{proof}
Clearly the sequence $(\alpha)$ starts with $\alpha=\mathbf0$ and is lexicographically increasing with respect to the standard basis.

Let $\alpha\ne0$ be a vector used in step \ref{recursion} at some iteration $q>1$ and let $\bar x$ be the last point computed at line \ref{bar x} before iteration $q$; 
say that $\bar x$ is computed at iteration $q'<q$. If $q'=q-1$, then $\alpha=C\bar x^\uparrow$ and therefore $\alpha\in V(S)$. If $q'=q-t$ for some $t>1$, then
line \ref{alpha1} is executed $t-1$ times between iterations $q'$ and $q$. In this case, $\alpha$ is the vector defined by $\alpha_i=c^i\bar x^\uparrow$ for
$i\le n-t$, $\alpha_{n-t+1}=c^{n-t+1}\bar x^\uparrow+1$, $\alpha_i=c^i\bar x^\uparrow$ for $i\ge n-t+2$, and therefore $\alpha \in V(S)$.

We now show that every point in $V(S)$ is in the sequence $(\alpha)$.
By Lemma \ref{lem:consecutive}, the sequence $(\alpha)$ contains all points of the form $Cx$ for $x\in S^\uparrow$.
Let now $\alpha\in V(S)$, where $\alpha$ is not of the form $Cx$ for any $x\in S^\uparrow$. Then there exist $\hat x\in S^\uparrow$ and an index $k<n$ such 
that $\alpha_i=c^i\hat x$ for $i<k$, $\alpha_kx=c^k\hat x+1$, and $\alpha_i=0$ for $i>k$. Consider the last iteration of line \ref{bar x} in which $\bar x^\uparrow$ 
satisfies $c^i\bar x^\uparrow=c^i\hat x$ for $i\le k$ (this definition makes sense because, as shown above, $\hat x=\bar x^\uparrow$ at some iteration of line \ref{bar x}). 
The algorithm now sets $\alpha=C\bar x^\uparrow$ and executes line \ref{alpha1} $k$ consecutive times. After this, we have $\alpha=Cx$. This shows that every point in $V(S)$ 
is in the sequence $(\alpha)$.
\end{proof}

We remark that in the definition of $\alpha$ at line \ref{alpha2}, we could impose the stronger condition $\alpha_n:=c^n\bar x^\uparrow+1$. However, this would not 
change substantially the bounds on the number of iterations shown above. Moreover, when $S$ is convex the number of iterations is precisely the same in both cases.

By Observation \ref{obs:performance} and Proposition \ref{prop:bound-alg2}, the number of iterations of Algorithm \ref{alg:weak} and Algorithm \ref{alg:enum} is 
upper-bounded by $|S^\uparrow|$ and $|V(S)|+1$, respectively. Note that the latter bound is always larger than the former. In particular, for the example in
Proposition \ref{prop:exp} we have $|V(S)|=2^n+2^{n-1}-2$, thus in that case Algorithm \ref{alg:enum} executes roughly 50\% more iterations than Algorithm \ref{alg:weak}.
However comparing the two algorithms by counting the number of iterations may not be ``fair'', as the computational effort varies from iteration to iteration: for instance, 
the computation of a lex-min solution (line \ref{step:lex-min} of Algorithm \ref{alg:weak} and line \ref{bar x} of Algorithm  \ref{alg:enum}) requires up to $n$ oracle calls, 
while the iterations of Algorithm  \ref{alg:enum} in which $S^*$ is empty only require a single oracle call. Nonetheless the results on the number of iterations at least indicate that,
from the theoretical point of view, Algorithm \ref{alg:weak} tends to be more efficient than Algorithm \ref{alg:enum}.

\section{Comparison with Gomory and split cuts}\label{sec:comp}

Given a set $S$, a {\em Chv\'atal--Gomory inequality} for $S$ is a linear inequality
of the form $gx\ge\ceil{\gamma}$ for some $g\in\Z^n$ and $\gamma\in\R$ such that the inequality $gx\ge\gamma$ is valid for $S$.
We call $gx\ge\ceil{\gamma}$ a {\em proper} Chv\'atal--Gomory inequality if $gx\ge\ceil{\gamma}$ is violated by at least one point in $S$.

\begin{proposition}
Given $S\in\mathcal S$, every proper Chv\'atal--Gomory inequality for $S$ is a lex-cut for some lattice basis $\{c^1,\dots,c^n\}$ of $\Z^n$.
\end{proposition}

\begin{proof}
Let $gx\ge\ceil{\gamma}$ be a proper Chv\'atal--Gomory inequality for $S$. Without loss of generality,
we assume that the entries of $g$ are relatively prime integers.
Let $\bar x$ be the lex-min solution found at the first iteration of Algorithm \ref{alg:weak}  with respect to some lattice basis $\{c^1,\dots,c^n\}$, with $c^1=g$.
Since $gx\ge\ceil{\gamma}$ is a proper Chv\'atal--Gomory inequality for $S$, we have $\gamma\le g\bar x<\ceil{\gamma}$. In particular, $g\bar x\notin\Z$.
Then the corresponding lex-cut is (equivalent to) $gx\ge\ceil{g\bar x}=\ceil{\gamma}$.
\end{proof}

The converse of the above proposition is false; this will follow from a stronger result.

A linear inequality is a {\em split cut} for $S$  if there  exist $\pi\in\Z^n$ and $\pi_0\in\Z$ such that 
the inequality is valid  for both  $\{x\in S:\pi x\le\pi_0\}$
and $\{x\in S:\pi x\ge\pi_0+1\}$.
It is known that every Chv\'atal--Gomory inequality is a split cut but not vice versa.

The next result shows that  our family of cuts is not included in and does not include the family of split cuts.
Combined with the previous proposition, this implies that our family of cuts strictly contains the  Chv\'atal--Gomory inequalities.

\begin{proposition}\label{prop:split}
There exist a bounded polyhedron $S$ and a split cut for $S$ that cannot be obtained as (and is not implied by) 
a lex-cut for any choice 
of the lattice basis $\{c^1,\dots,c^n\}$.
Conversely, there exist a bounded polyhedron $S$ and a lex-cut that is not a split cut for $S$.
\end{proposition}

\begin{proof}
Let $S\subseteq\R^2$ be the triangle with vertices $(0,0)$, $(1,0)$ and $(1/2,-1)$. (See Figure \ref{fig:split1} to follow the proof.) 
The inequality $x_2\ge0$ is a 
split cut for $S$, as it is valid for both sets $\{x\in S:x_1\le0\}$ and $\{x\in S:x_1\ge1\}$. Note that after the application of the cut,
the continuous relaxation 
becomes the segment with endpoints $(0,0)$ and $(1,0)$, which is the convex hull of the integer points in $S$.

Assume that the cut $x_2\ge0$ can be obtained via an iteration of Algorithm \ref{alg:weak} for some lattice basis $\{c^1,c^2\}$ and the 
corresponding bounds $\ell_1,\ell_2\in\Z$.
In the following, we will write $c^1=(c^1_1,c^1_2)$ and $c^2=(c^2_1,c^2_2)$.

Recall that in Algorithm \ref{alg:weak} a translation is applied such that $\ell_i=0$ for every $i$. However, in this proof it 
is convenient to work without applying 
the translation. It is easy to see that in this case the form of the lex-cut is still \eqref{eq:cut}, but now the $d^k_i$ are defined as follows:
\[d^k_i:=\begin{cases}
1 & \mbox{if $i=k$}\\
\ceil{c^k\bar{x}-\ell_k} & \mbox{if $i=k-1$},\\
\ceil{c^k\bar{x}-\ell_k}\prod_{j=i+1}^{k-1}(c^j\bar{x}+1-\ell_j),&\mbox{if $i\le k-2$}.
\end{cases}\]

Since the point $(1/2,-1)$ is the only fractional vertex of $S$, we must have $\bar x=(1/2,-1)$, otherwise no cut is generated.
Suppose $k=1$, i.e., $c^1\bar x\notin\Z$ (see step \ref{step:cut} of the algorithm). Then the inequality generated by the algorithm is equivalent to
$c^1x\ge\ceil{c^1\bar x}$. Since this inequality must be equivalent to $x_2\ge0$ and the entries of $c^1$ are relatively prime integers,
we necessarily have $c^1=(0,1)$. But then $c^1\bar x=-1$, a contradiction to the assumption $c^1\bar x\notin\Z$.

Suppose now $k=2$, i.e., $c^1\bar x\in\Z$ and $c^2\bar x\notin\Z$.
Then the inequality given by the algorithm is
\begin{equation}\label{eq:nonsplit}
d^2_1\left(c^1x-c^1\bar x\right)+c^2x-\ceil{c^2\bar x}\ge0.
\end{equation}

We claim that $c^1_1\ne0$. If this is not the case, then $c^1_1=0$ and $c^2_1\ne0$ (as $\{c^1,c^2\}$ is a basis), and inequality \eqref{eq:nonsplit} does not reduce 
to the desired cut $x_2\ge0$, as the coefficient of $x_1$ is $d^2_1c^1_1+c^2_1=c^2_1\ne0$. Thus $c^1_1\ne0$.
This implies that either the point $(0,-1)$ or the point $(1,-1)$ satisfies the strict inequality $c^1x>c^1\bar x$. We assume that this holds for $\hat x:=(0,-1)$ (the
other case is similar). Note that $c^1\hat x\ge c^1\bar x+1$, as $c^1\bar x\in\Z$ and $c^1,\hat x\in\Z^2$.
Furthermore, the slope of the line defined by the equation $c^1x=c^1\bar x$ is positive.

If $c^2\hat x\ge \ell_2$, then $\hat x$ satisfies inequality \eqref{eq:nonsplit}, as $c^1\hat x-c^1\bar x\ge1$ and $c^2\hat x-c^2\bar x\ge \ell_2-c^2\bar x\ge-d^2_1$.
Since the point $(1,0)$ also satisfies \eqref{eq:nonsplit} (as it is an integer point in $S$),
the middle point of $\hat x$ and $(1,0)$ satisfies \eqref{eq:nonsplit}. However, the middle point is $(1/2,-1/2)$, which is in $S$.
This shows that in this case \eqref{eq:nonsplit} is not equivalent to $x_2\ge0$.

Therefore we assume $c^2\hat x<\ell_2$. Since $c^2\bar x\ge \ell_2$, the line defined by
the equation $c^2x=\ell_2$ intersects the line segments $[\hat x,\bar x]$ in a point distinct from $\hat x$.
Then, because $(0,0)$ satisfies the inequality $c^2x\ge \ell_2$ (as it is in $S$), the slope of the line defined by the equation $c^2x=\ell_2$ is negative.
Furthermore, since $c^2,\hat x\in\Z^2$, we have $c^2\hat x\le\floor{\ell_2}$, and thus the line defined by the equation $c^2x=\floor{\ell_2}$ intersects $[\hat x,\bar x]$
in some point $x^*$.

Now consider the system $c^1x=c^1\bar x$, $c^2x=\floor{\ell_2}$. Since the constraint matrix is unimodular (as $\{c^1,c^2\}$ is a
lattice basis of $\Z^2$) and the right-hand sides are integer, the unique solution to this system is an integer point. However,
the first equation defines a line with positive slope containing $\bar x$ and the second equation defines a line with negative slope containing $x^*$.
From this we see that the intersection of the two lines is a point satisfying $0<x_1\le1/2$ and therefore cannot be an integer point, a contradiction.
This concludes the proof that there is a split cut that cannot be obtained via an iteration of Algorithm \ref{alg:weak}.

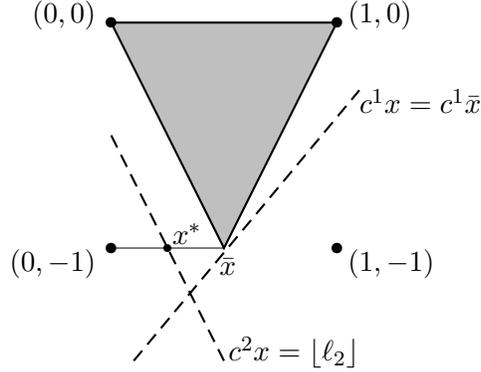
\begin{figure}
\begin{center}
\psset{unit=3cm}
\begin{pspicture}(0,-1.5)(1,0)
\pspolygon[fillstyle=solid,fillcolor=lightgray](0,0)(1,0)(.5,-1)
\put(0,0){\circle*{.05}}
\put(1,0){\circle*{.05}}
\put(0,-1){\circle*{.05}}
\put(1,-1){\circle*{.05}}
\put(-.35,0){$(0,0)$}
\put(1.05,0){$(1,0)$}
\put(-.45,-1.1){$(0,-1)$}
\put(1.05,-1.1){$(1,-1)$}
\put(.48,-1.12){$\bar x$}
\psline[linewidth=.1pt](0,-1)(.5,-1)
\psline[linestyle=dashed](.1,-1.5)(1.1,-.3)
\psline[linestyle=dashed](0,-.5)(.5,-1.5)
\put(.27,-.98){$x^*$}
\put(.25,-1){\circle*{.03}}
\put(1.1,-.4){$c^1x=c^1\bar x$}
\put(.52,-1.5){$c^2x=\floor{\ell_2}$}
\end{pspicture}
\end{center}
\caption{Illustration of the first part of the proof of Proposition \ref{prop:split}. The inequality $x_2\ge0$ is a split cut for the 
shadowed triangle, 
but is not of the type \eqref{eq:cut}.}
\label{fig:split1}
\end{figure}

\medskip
For the converse, let $S\subseteq\R^2$ be the triangle with vertices $(0,3/2)$, $(1/4,0)$ and $(1,0)$.
If we take $c^1,c^2$ to be the vectors in the standard basis of $\R^2$, and $\ell_1=\ell_2=0$,
then  Algorithm \ref{alg:weak} yields the cut $2x_1+x_2\ge2$.
Note that every point in $S$ other than $(1,0)$ is cut off by this inequality.
Thus, if the inequality $2x_1+x_2\ge2$ is a split cut for $S$, then there exist $\pi\in\Z^2$ and $\pi_0\in\Z$ such that $S$ is contained in the ``strip'' $\{x\in\R^2:\pi_0\le\pi x\le\pi_0+1\}$. 
Since $S$ contains a horizontal and a vertical segment of length $3/4$, this is possible only if the Euclidean distance between the lines $\{x\in\R^2:\pi x=\pi_0\}$
and $\{x\in\R^2:\pi x=\pi_0+1\}$ is at least $\frac3{4\sqrt2}$. Therefore $\|\pi\|^2\le\left(\frac{4\sqrt2}3\right)^2=\frac{32}{9}<4$.
Since $\pi$ is an integer vector, we deduce that $\pi_1,\pi_2\in\{0,1,-1\}$. It can be verified that if $|\pi_1|=|\pi_2|=1$ then $S$ is not contained in the strip.
Therefore one entry of $\pi$ is 0 and the other is 1 or $-1$. It can be checked that the only strip of this type containing $S$ is $\{x\in\R^2:0\le x_1\le 1\}$. However, 
the inequality $2x_1+x_2\ge2$ is not valid for all the points in $\{x\in S:x_1\le0\}\cup\{x\in S:x_1\ge1\}$, as the point $(0,3/2)$ is in this set but violates the inequality.
\end{proof}

\section{Concluding remarks}\label{sec:conc}
An obvious variant of Algorithm \ref{alg:weak} is the following: instead of being computed only once at the beginning of the procedure, 
the lower bounds $\ell_i$ can be updated 
at every iteration or whenever it seems convenient. It can be verified that the  bounds of Observation \ref{obs:performance} and 
Proposition \ref{prop:exp} also hold for this 
variant of the algorithm: the proofs are the same.

In view of  Observation \ref{obs:performance} and Proposition \ref{prop:exp}, the cardinality of $S^{\uparrow}$ truncated to the lex-min point 
in $S^{\uparrow}\cap S$ plays a 
crucial role in the performance of Algorithm \ref{alg:weak}. This number is dependent on the choice of the lattice basis and its ordering.
It is easy to see that different choices of the lattice basis (or different choices of the ordering of the elements of the same lattice basis)
may result in a different number of iterations of the algorithm. However, this is not always the case: for instance, in the example in 
Proposition \ref{prop:exp} Algorithm \ref{alg:weak} 
would produce the same number of iterations regardless of the ordering of the standard basis.

A natural question is whether the approach described in this paper can be generalized to the mixed integer case, i.e., 
to problems of the form $\min \{cx+dy: (x,y)\in S\cap(\Z^n\times\R^p)\}$, 
where $S\subseteq\R^{n+p}$ is a compact set.
However, it does not seem that our algorithm can be easily extended to deal with this case.

\section{Acknowledgements}
 The lex-cuts defined in a  previous version of this manuscript were weaker. Giacomo Zambelli suggested to derive stronger 
lex-cuts via the characterization of the polyhedron $Q(\bar{x})$. Section \ref{sec:exp} benefited from remarks due to Stefan Weltge. We thank both of them.
We are also grateful to Akshay Gupte for his constructive comments and his pointers to the existing literature.\medskip

\bibliographystyle{spmpsci} 
\bibliography{MINLPCuts}

\end{document}